\documentclass[12pt]{amsart}

\usepackage{times,amsfonts,amsmath,amstext,amsbsy,amssymb,
  amsopn,amsthm,upref,eucal, amscd,color,url,caption}
\usepackage[T1]{fontenc}
\usepackage[francais, english]{babel}
\usepackage[pdftex]{graphicx}
\usepackage{srcltx}

\newtheorem{defi}{\bf Definition}
\newtheorem{teo}{\bf Theorem}[section]
\newtheorem{lema}{\bf Lemma}[section]

\newtheorem{prop}{\bf Proposition}[section]
\newtheorem{cor}{\bf Corollary}[section]
\newtheorem{obs}{\bf Remark}[section]

\newtheorem{afir}{\bf Claim}
\newcommand{\La}{\Lambda}
\newcommand{\N}{\mathbb N}
\newcommand{\Z}{\mathbb Z}

\newcommand{\la}{\lambda}
\newcommand{\al}{\alpha}

\newcommand{\DD}[3]{\|Df^{#1}_{/E({#2})}\|\cdot\|Df^{-#1}_{/F(f^{#1}({#2}))}\|^{#3}}

\newcommand{\e}{\epsilon}

\newcommand{\ga}{\gamma}

\def\ZZ{\mathbb{Z}}

\newcommand{\emp}{\emptyset}

\newcommand{\noi}{\noindent}

\begin{document}

\title[Dynamics of Benedicks-Carleson toy models]{Axiom A versus Newhouse phenomena for Benedicks-Carleson toy models \\ (Axiome A versus ph\'enom\`ene de Newhouse pour les mod\`eles jouets de Benedicks-Carleson)}

\author{Carlos Matheus}
\address{Carlos Matheus: Universit\'e Paris 13, Sorbonne Paris Cit\'e, LAGA, CNRS (UMR 7539), F-93439, Villetaneuse, France.}
\email{matheus@impa.br.}
\urladdr{http://www.impa.br/$\sim$cmateus}

\author{Carlos G. Moreira} 
\address{Carlos G. Moreira: Instituto de Matem\'atica Pura e Aplicada (IMPA), Estrada D. Castorina, 110, J. Bot\^anico, 22460-320, Rio de Janeiro, Brazil.}
\email{gugu@impa.br.}
\urladdr{http://www.impa.br/$\sim$gugu}

\author{Enrique R. Pujals}
\address{Enrique Pujals: Instituto de Matem\'atica Pura e Aplicada (IMPA), Estrada D. Castorina, 110, J. Bot\^anico, 22460-320, Rio de Janeiro, Brazil.}
\email{enrique@impa.br.}
\urladdr{http://www.impa.br/$\sim$enrique}

\begin{abstract} We consider a family of planar systems introduced in 1991 by Benedicks and Carleson as a toy model for the dynamics of the so-called H\'enon maps. We show that Smale's Axiom A property is $C^1$-dense among the systems in this family, despite the existence of $C^2$-open subsets (closely related to the so-called Newhouse phenomena) where Smale's Axiom A is violated. In particular, this provides some evidence towards Smale's conjecture that Axiom A is a $C^1$-dense property among surface diffeomorphisms.

The basic tools in the proof of this result are: 1) a recent theorem of Moreira saying that stable intersections of dynamical Cantor sets (one of the main obstructions to Axiom A property for surface diffeomorphisms) can be destroyed by $C^1$-perturbations; 2) the good geometry of the dynamical critical set (in the sense of Rodriguez-Hertz and Pujals) thanks to the particular form of Benedicks-Carleson toy models.
\end{abstract}


\keywords{Axiom A, Newhouse phenomena, Benedicks-Carleson toy models, H\'enon maps, dynamical critical points, stable intersections of dynamical Cantor sets, two-dimensional dynamical systems. \\ .\hspace{0.2cm} \emph{Mots-cl\'es.} Axiome A, ph\'enom\`ene de Newhouse, mod\`eles jouets de Benedicks-Carleson, applications d'H\'enon, points critiques dynamiques, intersections stables des ensembles de Cantor dynamiques, syst\`emes dynamiques en dimension deux.}

\date{May 14, 2013}

\bibliographystyle{alpha}

\maketitle

\vspace{-0.9cm}

\selectlanguage{francais}
\begin{abstract} Nous consid\'erons une famille de syst\`emes introduite en 1991 par Benedicks et Carleson comme un mod\`ele jouet pour la dynamique des applications d'H\'enon. Nous montrons que l'axiome A de Smale est une propri\'et\'e $C^1$-dense parmi les syst\`emes dans cette famille, m\^eme si nous trouvons aussi des ensembles $C^2$-ouverts (li\'es au ph\'enom\`ene de Newhouse) o\`u l'axiome A de Smale n'est pas satisfait. En particulier, notre r\'esultat soutient la conjecture de Smale selon laquelle l'axiome A est une propri\'et\'e $C^1$-dense parmi les diff\'eomorphismes de surfaces. 

Les outils utilis\'es dans la preuve de notre r\'esultat sont: 1) un th\'eor\`eme r\'ecent de Moreira qui dit que  les intersections stables des ensembles de Cantor dynamiques (une des obstructions majeures \`a l'axiome A pour les diff\'eomorphismes de surfaces) peuvent \^etre enlev\'ees par des perturbations $C^1$-petites; 2) la bonne g\'eom\'etrie de l'ensemble de points critiques dynamiques (au sens de Rodriguez-Hertz et Pujals) due \`a la forme particuli\`ere des mod\`eles jouets de Benedicks-Carleson. 
\end{abstract}
\selectlanguage{english}


\section{Introduction}

Uniform hyperbolicity (Smale's Axiom A property) has been a long standing paradigm of complete
dynamical description: any dynamical system such that the tangent bundle over
its limit set (i.e., the set of accumulation points of all orbits) splits into two
complementary subbundles which are uniformly forward (respectively
backward) contracted by the tangent map can be completely described
from a geometrical and topological point of view.

Nevertheless, uniform hyperbolicity is a property less universal
than it was initially thought: {\em there are non-empty open sets in the space
of dynamics containing only non-hyperbolic systems.} Actually, Newhouse showed
that for smooth surface diffeomorphisms, the unfolding of a {\em
homoclinic tangency} (a non transversal intersection of stable and
unstable manifolds of a periodic point) generates  non-empty open sets of
diffeomorphisms whose  limit sets are non-hyperbolic (see
\cite{N1}, \cite{N2}, \cite {N3}).



It is important to say that a homoclinic tangency is (locally)
easily destroyed by small perturbation of the invariant manifolds.
To get open sets of diffeomorphisms with persistent homoclinic
tangencies, Newhouse considers certain systems where the homoclinic
tangency is associated to an invariant hyperbolic set with large
fractal dimension. In particular, he studied the intersection of the local
stable and unstable manifolds of a hyperbolic set (for instance, a
classical horseshoe), which, roughly speaking, can be visualized as
a product of two Cantor sets whose thicknesses are large. Newhouse's
construction depends on how this fractal invariant varies with
perturbations of the dynamics, and actually this is the main reason
that his construction works in the $C^2-$topology. In fact, Newhouse
argument is based on the continuous dependence of the thickness with
respect to $C^2$ perturbations. A similar construction in the
$C^1-$topology leading to same phenomena is unknown (indeed, some
results in the \emph{opposite} direction can be found in \cite{U} and \cite{M}).
In this setting, denoting by $\textrm{Diff}^r(M^n)$ the set of $C^r$-diffeomorphisms of a 
compact $n$-dimensional manifold $M^n$ (without boundary), it was \emph{implicitly} conjectured by Smale (cf. \cite{Smale}, Problems (6.10), item (a), at page 779) that

\begin{center}
{\it Axiom A surface diffeomorphisms are $C^1$ open and dense in} $\textrm{Diff}^1(M^2)$.
\end{center}
This question is explicitly called \emph{Smale's conjecture} in \cite{ABCD}.

In the present paper, we consider a special set of maps acting on a
two dimensional rectangle, firstly introduced by Benedicks and Carleson as a \emph{toy model} for the so-called \emph{H\'enon maps}. 
For this special type of systems, we show
that, if one deals with $C^2-$topology, there are non-empty  open sets of
diffeomorphisms which are not hyperbolic, while in the
$C^1-$topology, the Axiom A property is open and dense.

Before proceeding further, let us briefly recall some features of H\'enon maps and Benedicks-Carleson toy models. 

A typical family where the Newhouse's phenomena hold is the so
called H\'enon maps. In fact, it was proved in \cite{U2} that,
for certain parameter of this family, the unfolding of a tangency
leads to a non-empty open set of non-hyperbolic diffeomorphisms. 

On the other hand, numerical simulations indicate that the attractor of the H\'enon map
(i.e., the closure of the unstable manifold of its fixed saddle
point) has the structure of the product of a line segment and a
Cantor set with small dimension (when a certain parameter $b$ is
close to zero). Although it is a great oversimplification (and many
of the later difficulties on the analysis of H\'enon attractors arise
because of the roughness of such approximation), this idea gives a
very good understanding of the geometry of the H\'enon map. As a guide
to what follows, it is worth to point out that Benedicks and
Carleson~\cite[Section 3, p.~89]{BC} have constructed a model where the point moves on a pure
product space $( - 1 , 1) \times  \mathcal{K}$ where $\mathcal{K}$ is the Cantor set
obtained by repeated iteration of the division proportions $(b,
1-2b, b)$ (i.e., $\mathcal{K}=\bigcap\limits_{n\geq0} A^{-n}([0,b]\cup [1-b,1])$ where $A|_{[0,b]}(x)=x/b$ and $A|_{[1-b,1]}(x)=(b-x)/b$), and the dynamics on $(-1,1)$ is given by a family of
quadratic maps: in fact, the dynamical system on $(-1,1)$ acts as a
movement on a fan of lines, where each line has its own
$x$-evolution, while it is contracted in the $y$-direction (see
Figure 1).

More precisely, consider a one parameter family $\{f(x,y)\}_{y\in
[0,1]}$ (here $x$ is the variable and $y$ is the parameter) such that, for each fixed parameter $y\in [0,1]$, $$f(.,y):[-1,1]\to [-1,1]$$ is a $C^r$-unimodal map (with respect
to the variable $x$) verifying that $0$ is the critical point and $f(0,y)$ is the maximum value of $f(.,y)$
for all $y\in[0,1]$. We denote by ${\mathcal U}^r$ the set of families of
$C^r$-unimodal maps satisfying the conditions stated above.

Let $k:[0,a]\cup [b,1]\to [0,1]$ be a $C^r$ function such that
$k(0)=0=k(1),\, k(a)=1=k(b)$ and $|k'|>\ga>1$. Put
$$K(x,y) = \left\{ \begin{array}{ll}
         K_+(y) & \mbox{if}\,\, x > 0,\\
        K_-(y) & \mbox{if}\,\, x < 0,\end{array} \right. $$
where $K_+=(k_{/[0,a]})^{-1}, K_-=(k_{/[b,1]})^{-1}$.

The bulk of this article is the study of the dynamics of Benedicks-Carleson toy models
$F: ([-1,1]\setminus\{0\})\times[0,1]\to [-1,1]\times [0,1]$ given
by
\begin{equation}\label{e.F}
F(x,y)=(f(x,y), K(x,y))=(f(x,y), K_{sgn(x)}(y)).
\end{equation}

\begin{figure}[!htb]
\begin{center}
\includegraphics[scale=0.5]{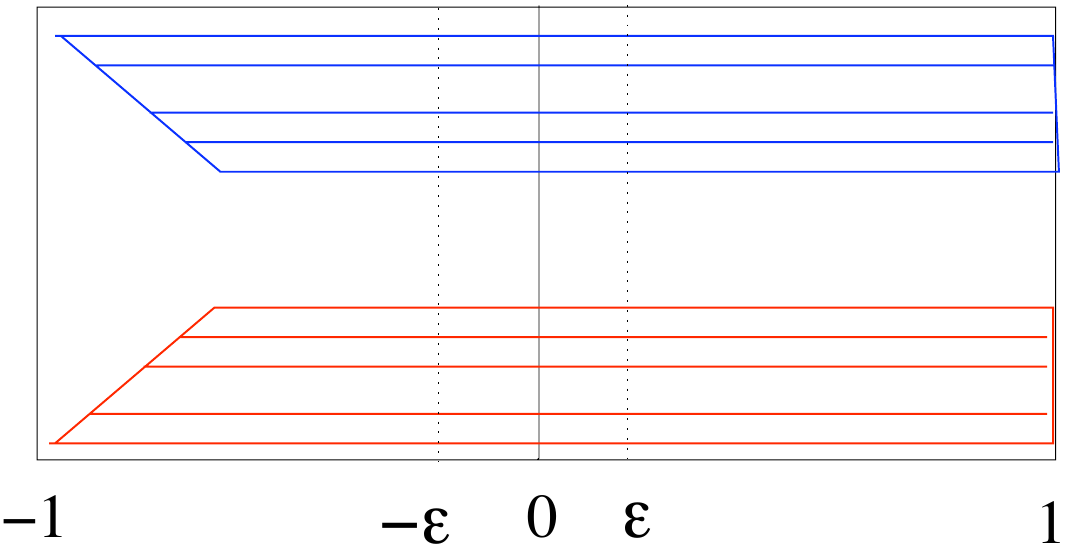}
\end{center}
\caption{Dynamics of $F(x,y)=(1-a(y)x^2, K_{sgn(x)}(y))$ with $a:[0,1]\to (0,2]$.}
\label{henon}
\end{figure}

\vskip 3pt

We denote by ${\mathcal D}^r$ the set of such maps $F$ (with $f(.,y)\in {\mathcal U}^r$ and $k\in C^r$) endowed with the $C^r$-topology. 
Since the line $\{x=0\}$ is a discontinuity line of any
$F\in {\mathcal D}^r$, the maps $F$ are $C^r$-diffeomorphisms only
on $([-1,1]-\{0\})\times [0,1]$, and we will
explain in Section \ref{s.preliminaries} the exact definition of the $C^r$-topology.
 Although this is not specially hard to do, we prefer to postpone it  (where we also revise the notion of hyperbolic sets of $F$) to avoid 
the appearance of unnecessary technicalities in this introductory section.

At this point, we are ready to state our main results:

\vskip 3pt

\noi {\bf Theorem A. } {\em  For $r\geq 2,$ there exists a non-empty open set
${\mathcal N}\subset {\mathcal D}^r$ such that no $F\in {\mathcal N}$ is Axiom A. Moreover, there exists a residual set $\mathcal{R}\subset\mathcal{N}$ 
such that any $F\in\mathcal{R}$ has infinitely many periodic sinks.}

\vskip 3pt

On the other hand, in the $C^1-$topology, the opposite statement
holds:

\vskip 3pt

\noi {\bf Theorem B. }{\em There exists an open and dense set ${\mathcal
V}\subset {\mathcal D}^1$ such that every $F\in {\mathcal V}$ is Axiom A.}

\vskip 3pt

Concerning the proof of these results, a fundamental role will be
played by certain points in the line $\{x=0\}$:

\begin{defi}\label{d.critical-point}Given $F\in {\mathcal D}^r$, consider $k:[0,a]\cup [b,1]\to
[0,1]$ the Cantor map related to $F$ and denote by $\mathcal{K}_0$ the Cantor
set induced by $k$. For any $y\in \mathcal{K}_0$, we call
$$c_y^{\pm} = (0^{\pm},y)$$
a \emph{critical point} of $F$.
\end{defi}  

From the technical point of view, it is important to introduce the points $c_y^{\pm}=(0^{\pm},y)$ because we can extend $F$ to them (via the formula $F(0^{\pm},y)=(f(0,y), K_{\pm}(y))$), so that $F$ becomes defined on a compact set. Of course, we have to pay the price that this extension of $F$ is no longer continuous. In particular, let us make a few comments about the orbits and the meaning of the non-wandering set $\Omega(F)$ of this extension of $F$. While the orbits do not intersect the line 
$\{x=0\}$, we have nothing to say. On the other hand, when the iterate $(0,y)=F(z,w)$ of a point $(z,w)\in([-1,1]-\{0\})\times [0,1]$ hits the line $\{x=0\}$, we will consider that \emph{both} points $(0^{\pm},y)$ make part of the orbit of $(z,w)$. Finally, we say that a point $(z,w)$ is non-wandering when any neighborhood $U$ of $(z,w)$ has some iterate $F^n(U)$ such that $F^n(U)\cap U\neq\emptyset$. Here, a small neighborhood of a point $(z,w)\in([-1,1]-\{0\})\times [0,1]$ is a small standard (Euclidean) neighborhood, while a small neighborhood of the point $(0^+,y)$, resp., $(0^-,y)$, is a ``half-neighborhood'' obtained from the intersection of $[0,1]\times [0,1]$, resp. $[-1,0]\times [0,1]$, with a small standard (Euclidean) neighborhood of $(0,y)$. 

The relevance of the concept of critical point becomes clear from the following
simple (but conceptually important) remark:

\begin{obs}\label{tang}It follows from the definition that, if $c_y^{\pm}\in \Omega(F)$ and $c_y^{\pm}$ is not a periodic sink,
then $\Omega(F)$ is not hyperbolic in the sense of Definition \ref{d.hyperbolic-set} below. This fact should be compared to the notion of \emph{dynamical critical points} of~\cite{PRH} and its role as the obstruction to the presence of hyperbolicity/domination in dissipative compact invariant sets of surface diffeomorphisms.
\end{obs}

Closing this introduction, we give the organization of the paper:

\begin{itemize}
\item In Section~\ref{s.A}, we follow the same ideas of Newhouse to construct a $C^2$-open set $\mathcal{N}$ where the critical points can not be removed from the limit set, so that the proof of Theorem A can be derived from the combination of this fact and Remark~\ref{tang}.
\item In Section~\ref{s.B}, the proof of Theorem B is presented. Morally speaking, our basic idea is inspired by a proof of Jakobson's theorem~\cite{J} (of $C^1$-density of hyperbolicity among unimodal maps of the interval) along the lines sketched in the book of de Melo and van Strien~\cite{dMvS}: namely, in the one-dimensional setting, one combines Ma\~n\'e's theorem~\cite{M1} (giving the hyperbolicity of compact invariant sets far away from critical points of a $C^2$ Kupka-Smale interval map) with an appropriate $C^1$-perturbation to force the critical point to fall into the basin of a periodic sink. In our two-dimensional setting, we start by showing that the points
of the limit set staying away from the critical line $\{x=0\}$
belong to a hyperbolic set; this is done by proving that any
compact set disjoint from the critical line exhibits a dominated
splitting and then by using Theorem B in \cite{PS1} (which is the two-dimensional generalization of Ma\~n\'e's theorem~\cite{M1}) to
conclude hyperbolicity. Next, we exploit a recent theorem of
Moreira~\cite{M} about the non-existence of $C^1$-stable
intersections of Cantor sets plus the geometry of the maps $F\in
{\mathcal D}^1$ to prove a dichotomy for the critical points of a
generic $F$: either critical points fall into the basins of a
finite number of periodic sinks or they return to some small
neighborhood of the critical line. Finally, we prove the critical
points returning close enough to the critical line can be absorbed
by the basins of a finite number of periodic sinks after a
$C^1$-perturbation; thus, we conclude that the limit set of a generic
$F\in {\mathcal D}^1$ is the union of an hyperbolic set with a finite
number of periodic sinks, i.e., a generic $F\in {\mathcal D}^1$ is Axiom
A.
\end{itemize}

\noindent\textbf{Acknowledgements.} The authors are thankful to IMPA, Coll\`ege de France and Institut Mittag-Leffler (and their staff) for the excellent ambient during the preparation of this manuscript. Also, we are grateful to Sylvain Crovisier for several discussions (who helped to clarify the arguments below). Moreover, we would like to acknowledge Sylvain Crovisier and Jean-Christophe Yoccoz for their interest in this work and their constant support.


\section{Preliminaries}\label{s.preliminaries}

\vskip 3pt

In this (very) short section, we quickly review a few technical notions appearing in the statements of Theorems A and B.

\begin{defi}Given $s\geq r\geq 1$ integers and $F, \widetilde{F} \in {\mathcal D}^s$, consider
$\{f(.,y)\}_{y\in [0,1]}$ and $k:[0,a]\cup [b,1]\to [0,1]$, respectively
$\{\widetilde{f}(.,y)\}_{y\in [0,1]}$ and
$\widetilde{k}:[0,\widetilde{a}]\cup [\widetilde{b},1]\to [0,1]$
 (the functions associated to $F$, respectively $\widetilde{F}$). 

We say that $F$ and $\widetilde{F}$ are $C^r$-close if the one parameter families $\{f(.,y)\}_{y\in
[0,1]}$ and $\{\widetilde{f}(.,y)\}_{y\in [0,1]}$ are $C^r$-close
in the usual manner, $a$ is close to $\widetilde{a}$, $b$ is close
to $\widetilde{b}$, and $k$ is $C^r$-close to $\widetilde{k}$ in the sense that they admit $C^s$-extensions to $[0,\max\{a,\widetilde{a}\}]\cup [\min\{b,\widetilde{b}\},1]$ which are $C^r$-close.
\end{defi}

\begin{defi}\label{d.hyperbolic-set}A set $\La$ is called \emph{hyperbolic} for
$F\in\mathcal{D}^r$ if it is compact, $F$-invariant and there exist a decomposition $\mathbb{R}^2=E^s\oplus
E^u$ invariant under $DF$ and some constants $C>0$, $0<\la<1$ such that
$$|DF^n_{/E^s(x)}|\le C\la^n\,\,\,\mbox{and}\,\,\,|DF^{-n}_{/E^u(x)}|\le
C\la^n\,\,\,\forall x\in \La,\,\,\, n\in \N.$$ 
Here, it is worth to point out that we do not require the dimensions of $E^s$ and $E^u$ to be constant on $\Lambda$ (contrary to some places in the literature). In particular, a hyperbolic set in our context is always the union of a saddle-type set (i.e., a hyperbolic set where $\textrm{dim}(E^s)=\textrm{dim}(E^u)=1$) disjoint from the set of critical points and finitely many periodic sinks.

We say that $F\in\mathcal{D}^r$ is \emph{Axiom A} if the non-wandering set is
hyperbolic and it is the closure of the periodic points. In the
sequel, $\Omega(F)$ denotes the non-wandering set (as defined in the paragraph right after Definition \ref{d.critical-point} above).
\end{defi}


\section{Proof of Theorem A}\label{s.A}

\vskip 3pt

The strategy is similar to the arguments of \cite{N1} (see
also~\cite{PT}).

Given $0<t<1$ and $m\geq m_0=m_0(t)$ (where $m_0(t)$ is a large integer to be chosen later), we define  $\delta_m:=1/(2^m-1)$, $\e_m:=\sin(\pi\delta_m/2)$ and we select a parameter $\rho_m$ such that $1-\cos(\pi\delta_m)< t\rho_m/2< 1-\cos(\pi(1-\delta_m)/2^{m-1})$ (e.g., $\rho_m:=2(1-\cos(3\pi\delta_m/2))/t$ works for $m_0(t)$ sufficiently large). Next, we take $\mu_m: [0,1]\to [0,1]$ a $C^2$-map such that $\mu_m(y)=\mu_m(1-y)$ and $\mu_m(y) = 1-\sqrt{1-\rho_m y/2}$ for every
$y\in [0,\frac{t}{2}]$ and we define
\begin{equation}\label{e.bc}
F^t(x,y)= (f_{\e_m}(x,y),K^t(x,y)),
\end{equation}
with $$K^t(x,y) := \left\{ \begin{array}{ll}
         (k^t_{/[0,\frac{t}{2}]})^{-1}(y) & \mbox{if}\,\, x > 0,\\
         (k^t_{/[1-\frac{t}{2},1]})^{-1}(y) & \mbox{if}\,\, x < 0,\end{array} \right.$$
where $k^t$ is the map
\begin{displaymath}
k^t(y)=\left\{ \begin{array}{ll}
2y/t & \mbox{if}\,\, 0\leq y\leq t/2 ,\\
2(1-y)/t& \mbox{if}\,\, 1-\frac{t}{2}\leq y\leq
1,\end{array} \right.
\end{displaymath}
and $f_{\e_m}(x,y)$ is a $C^2$ family of unimodal maps such that
$$f_{\e_m}(x,y)=\left\{ \begin{array}{ll}
         1-2x^2 & \mbox{if}\,\, |x| \geq \e_m,\\
         1-\mu_m(y) & \mbox{at}\,\, x=0.\end{array} \right.$$
Also, let $\mathcal{K}_0=\mathcal{K}^t_0:=\cap_{n\in \N} (k^t)^{-n}([0,\frac{t}{2}]\cup [1- \frac{t}{2}, 1])$ be the Cantor set induced by $k=k^t$. See Figure 2.

\begin{figure}[!htb]
\begin{center}
\includegraphics[scale=0.5]{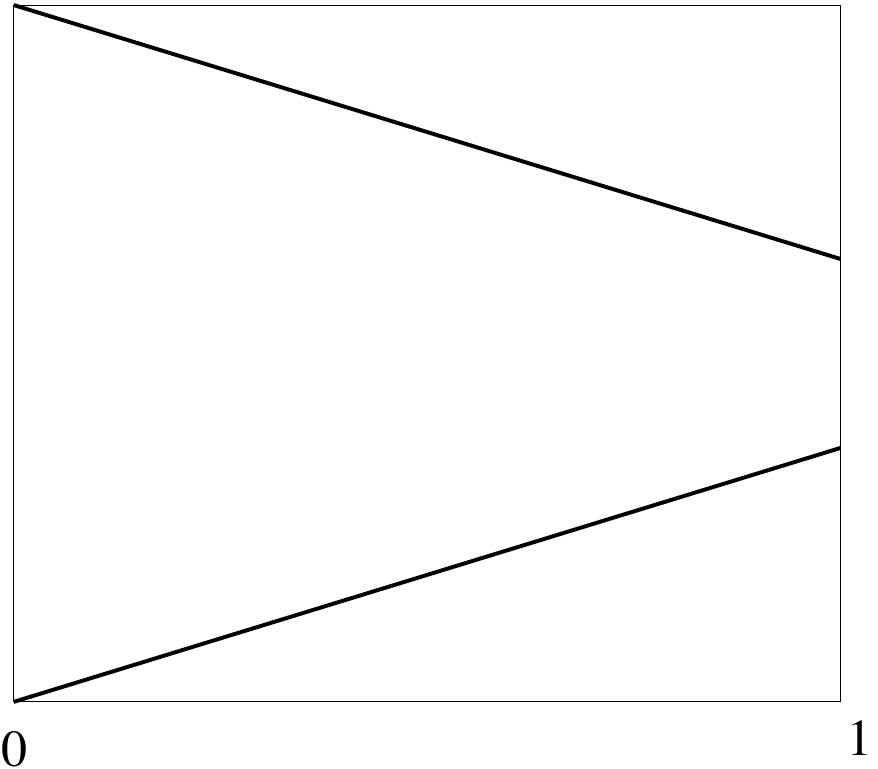}
 \\ [0.4cm]
$y\mapsto \frac{t}{2}y,\,\,x>0;\,\,\, y\mapsto 1-\frac{t}{2}y,\,\,x<0$
\end{center}
\caption{Dynamics of $K^t$.}
\label{figtest-fig}
\end{figure}

To simplify the exposition, firstly we consider the proof of Theorem
A only for maps $F=F^t$ of the form~(\ref{e.bc}). Then, at the end of this
section, we explain how
the general case follows from the previous one.

We begin by recalling some classical facts about dynamically defined Cantor sets and their thicknesses. For a more detailed explanation, see~\cite{PT}.

\begin{defi}
We say that a Cantor set $\mathcal{K}\subset\mathbb{R}$ is
\emph{dynamically defined} if it is the maximal invariant set of a
$C^{1+\al}$ expanding map with respect to a given Markov partition.
\end{defi}

\begin{defi}A \emph{gap} (resp. \emph{bounded gap}) of a Cantor set $\mathcal{K}$ is a
connected component (resp., bounded connected component) of $\mathbb{R} - \mathcal{K}$. Given $U$ a bounded gap of $\mathcal{K}$
and $u\in\partial U$, we call the \emph{bridge} $C$ of $\mathcal{K}$
at $u$ to the maximal interval such that $u\in\partial C$ and $C$
contains no point of a gap $U'$ with $|U'|\geq |U|$. The thickness
of $\mathcal{K}$ at $u$ is $\tau(\mathcal{K},u) = |C|/|U|$ and the
\emph{thickness} $\tau(\mathcal{K})$ of $\mathcal{K}$ is the infimum
over $\tau(\mathcal{K},u)$ for all boundary points $u$ of bounded
gaps.
\end{defi}

\begin{obs}\label{th} For the Cantor sets $\mathcal{K}_0^t$ induced by the maps $k^t$ above, it
is not hard to see that $0<\tau(\mathcal{K}_0^t)=t/2(1-t)<\infty$.
\end{obs}

\begin{obs}\label{tent} The quadratic map $f_2(x):=1-2x^2$ has arbitrarily thick dynamically defined Cantor sets. In fact, using the fact that $1-2x^2$ is conjugated to the complete \emph{tent} map
$$T_2(x):=\left\{ \begin{array}{ll}
         2x & \mbox{if}\,\, 0\leq x\leq 1/2 ,\\
         2-2x & \mbox{if}\,\, 1/2\leq x\leq 1,\end{array} \right.$$
via the \emph{explicit} conjugation $h(x)=-\cos(\pi x)$, we can exhibit thick Cantor sets as follows. Denote by $\widetilde{I}_2^{(m)}:=[h(2\delta_m), h((1-\delta_m)/2^{m-2})]$ and put $\widetilde{I}_i^{(m)}:=f_2(\widetilde{I}_{i-1}^{(m)})$ for $i=3,\dots,m$. As it is explained in Section 2 of Chapter 6 of Palis-Takens book~\cite{PT}, from the explicit nature of the conjugation $h$ and the fact that the intervals $h^{-1}(\widetilde{I}_2^{(m)}),\dots,h^{-1}(\widetilde{I}_m^{(m)})$ form a Markov partition of a dynamically defined Cantor set $\mathcal{K}_m$ of thickness $\tau(\mathcal{K}_m)=2^{m-1}-3$ associated to the tent map $T_2(x)$, it is possible to check that 
$\widetilde{\mathcal{K}}_m:=h(\mathcal{K}_m)$ are dynamically defined Cantor sets associated to $f_2$ (and Markov partition $\widetilde{I}_2^{(m)},\dots,\widetilde{I}_m^{(m)}$) such that $\tau(\widetilde{\mathcal{K}}_m)\to\infty$ (as $m\to\infty$).
\end{obs}

\begin{obs}\label{cont}
Let $\mathcal{K}(\psi)$ be the dynamically defined Cantor set
associated to a $C^{1+\al}$ expanding map $\psi$. If $\phi$ is
$C^{1+\al}$-close to $\psi$, then the thickness of
$\mathcal{K}(\phi)$ is close to the thickness of
$\mathcal{K}(\psi)$. In other words, the thickness of dynamically
defined Cantor sets $\mathcal{K}$ depend continuously on
$\mathcal{K}$ (with respect to the $C^{1+\alpha}$-topology). See~\cite{PT}.
\end{obs}

Now we state Newhouse's \emph{gap lemma} ensuring that two
\emph{linked} Cantor sets with large thicknesses should intersect
somewhere:

\begin{lema}[Gap Lemma~\cite{N1}]\label{gap}Given two Cantor sets $\mathcal{K}_1$ and
$\mathcal{K}_2$ of $\mathbb{R}$ such that
$$\tau(\mathcal{K}_1)\tau(\mathcal{K}_2)>1,$$ then one of the
following possibilities occurs:
\begin{itemize}
\item $\mathcal{K}_1$ is contained in a gap of $\mathcal{K}_2$;
\item $\mathcal{K}_2$ is contained in a gap of $\mathcal{K}_1$;
\item $\mathcal{K}_1\cap \mathcal{K}_2 \neq \emptyset$.
\end{itemize}
\end{lema}

For later reference, we recall the following definition:

\begin{defi}
We say that $\mathcal{K}_1$ and $\mathcal{K}_2$ are \emph{linked} if their convex hulls 
$I_1$ and $I_2$ are linked in the sense that the interior of $I_2$ contains exactly one boundary point of $I_1$ and vice-versa.
\end{defi}
Observe that the ``linked'' property is robust by perturbations.
After these preliminaries, we can complete the discussion of this section as follows.

\vskip 4pt

\noi {\bf End of the proof of Theorem A:}

\vskip 3pt

We observe that, since $F=F^t$ is the product map $F^t(x,y)=(1-2x^2,K_{sgn(x)}^t(y))$ at the region $([-1,\e_m]\cup[\e_m,1])\times[0,1]$, it follows that $\La_{\e_m} := \widetilde{\mathcal{K}}_m\times \mathcal{K}_0^t$ is a hyperbolic set of $F^t$. Moreover, the stable lamination $W^s(\La_{\e_m})$ is composed by vertical lines passing through $\widetilde{\mathcal{K}}_m\times\{0\}$ and the unstable lamination $W^u(\La_{\e_m})$ is composed by horizontal lines passing through $\{0\}\times \mathcal{K}_0^t$. We divide the construction of $\mathcal{N}$ into three steps.

\vskip 3pt

{\it Step 1:} From Remarks~\ref{th} and~\ref{tent}, given $0<t<1$,
we can choose $m_0(t)\in\mathbb{N}$ large such that, for every $m\geq m_0(t)$, it holds $$\tau(\widetilde{\mathcal{K}}_m)\tau(\mathcal{K}_0^t)>2.$$ \vskip 3pt

{\it Step 2:} Consider the following line segment:
$$L^+ := F^2(\{0^+\}\times [0,t/2])=\{(1-2(1-\mu_m(y))^2, \frac{t^2}{4}y)\}_{y\in [0,t/2]}=\{(-1+\rho_m y, \frac{t^2}{4}y)\}_{y\in [0,t/2]}.$$
In the sequel, $L^+$ plays the role of a \emph{line of tangencies}: more precisely, we introduce $$\tilde{\mathcal{K}^s}=
(f_2^{-1}(\widetilde{I}_2^{(m)}\cap\widetilde{\mathcal{K}}_m)\times [0,1])\cap L^+, \quad \quad \tilde{\mathcal{K}^u}=F^2(\{0^+\}\times (\mathcal{K}^t_0\cap [0,t/2])) = W^u_{loc}(\La_{\e_m})\cap L^+.$$
Here, we used the fact that the critical point $0$ of $f_{\e_m}$ belongs to the interval $\widetilde{I}_m^{(m)}$ of the Markov partition of the dynamical Cantor set $\widetilde{\mathcal{K}}_m$ (and, thus, the preimages of $0$ accumulate $\widetilde{\mathcal{K}}_m$) to get that the critical points $c^{\pm}_y$ belong to $W^u(\Lambda_{\e_m})$.

We claim that $\tilde{\mathcal{K}^s}\cap\tilde{\mathcal{K}^u}\neq\emptyset$. In fact, since the straight line segment $L^+$ is transversal to both horizontal and vertical foliations, and $L^+$ is naturally identified with the interval $[0,t/2]$ of $\mathbb{R}$ via $L^+\ni (-1+\rho_m y, \frac{t^2}{4}y)\mapsto y\in\mathbb{R}$, we obtain that $\tau(\tilde{\mathcal{K}^s}) \geq \tau(\widetilde{\mathcal{K}}_m)/2$ (as the derivative of $f_2$ in the interval $[0,t/2]$ is between $1$ and $2$) and $\tau(\tilde{\mathcal{K}^u})=\tau(\mathcal{K}_0^t)$, so that $\tau(\tilde{\mathcal{K}^s})\tau(\tilde{\mathcal{K}^u})>1$ (by Step 1). Hence, by Newhouse gap lemma~\ref{gap}, it suffices to show that $\tilde{\mathcal{K}^s}$ and $\tilde{\mathcal{K}^u}$ are linked. However, it is not hard to see that this follows from our choice of $\rho_m$. Indeed, from the definitions of $\tilde{\mathcal{K}^s}$ and $\tilde{\mathcal{K}^u}$, we get that $\tilde{\mathcal{K}^s}$ and $\tilde{\mathcal{K}^u}$ are linked if and only if the vertical projection $\overline{\mathcal{K}}^s:= f_2^{-1}(\widetilde{I}_2^{(m)}\cap\widetilde{\mathcal{K}}_m)$ of $\tilde{\mathcal{K}^s}$ is linked to the vertical projection $\overline{\mathcal{K}}^u$ of $\tilde{\mathcal{K}^u}$. On the other hand, the convex hulls of $\overline{\mathcal{K}}^s$ and $\overline{\mathcal{K}}^u$ are linked: more precisely, the convex hull $I^s$ of $\overline{K}^s$ is $f_2^{-1}(\widetilde{I}_2^{(m)}) = [-\cos(\pi\delta_m),-\cos(\pi(1-\delta_m)/2^{m-1})]$ and the convex hull $I^u$ of $\overline{K}^u$ is $[0,-1+t\rho_m/2]$, so that our choice of $\rho_m$ verifying
$$1-\cos(\pi\delta_m)< t\rho_m/2< 1-\cos(\pi(1-\delta_m)/2^{m-1})$$
implies that $I^s$ and $I^u$ are linked.

Next, we notice that $\tilde{\mathcal{K}^s}\cap \tilde{\mathcal{K}^u}\neq\emptyset$ means that $F^2(c_y^+)\in W^s_{loc}(\La_{\e_m})$ for some critical point $c^+_y\in W^u_{loc}(\Lambda_{\e_m})$, $y\in \mathcal{K}_0^t$. Since the hyperbolic set $\Lambda_{\e_m}$ is transitive (i.e., it contains dense orbits), it follows that $c_y^+$ is a non-periodic critical point belonging to the non-wandering set $\Omega(F)$. Therefore, by Remark
\ref{tang}, the set $\Omega(F)$ is not hyperbolic. \vskip 3pt

{\it Step 3:} Finally, we claim that any sufficiently small $C^2$ neighborhood $\mathcal{N}\subset \mathcal{D}^2$ of the map $F=F^t$ constructed above fits the conclusion of the first part of Theorem A. Indeed, this is a consequence of the following known facts for $G$ $C^2$-close to $F$:

\begin{enumerate}
\item The hyperbolic set $\La_{\e_m}$ has a continuation to a hyperbolic set $\La_{\e_m}(G)$ of $G$;
\item The Cantor sets $\tilde K^s$ and $\tilde K^u$ have unique continuations to Cantor sets
$\tilde K^s(G)$ and $\tilde K^u(G)$ obtained by intersecting the local stable and unstable laminations of $\La_{\e_m}(G)$ with the \emph{line of tangencies} $L^+(G)=G^2(\{0^+\}\times [0,1/2])$. Moreover, these Cantor sets are
$C^{1+\al}-$close to $\tilde K^s$ and $\tilde K^u$ respectively in the sense that their vertical projections $\overline{K}^s(G)$ and $\overline{K}^u(G)$ to $\mathbb{R}\times\{0\}$ are $C^{1+\alpha}$-close to the vertical projections $\overline{K}^s$ and $\overline{K}^u$ of $\tilde K^s$ and $\tilde K^u$;
\item Thus, the Cantor sets $\overline{K}^s(G)$ and $\overline{K}^u(G)$ have thicknesses close to the
thicknesses of $\overline{K}^s$ and $\overline{K}^u$ respectively; by continuity of the thickness (see Remark~\ref{cont}), it follows that $\tau(\overline{K}^u(G))\tau(\overline{K}^u(G))>1$;
\item From Newhouse gap lemma \ref{gap} and the fact that $\overline{K}^s(G)$ and $\overline{K}^u(G)$ remain linked, it follows that $\overline{K}^s(G)\cap \overline{K}^u(G)\neq \emp$ and, \emph{a fortiori}, $\tilde K^s(G)\cap \tilde K^u(G)\neq \emp$;
\item Hence, there are (non-periodic) critical points contained
in the non-wandering set of $G$, and so, by Remark~\ref{tang}, it is not hyperbolic.
\end{enumerate}

At this point, it remains only to prove the second part of Theorem A, namely, the existence of a residual set $\mathcal{R}\subset\mathcal{N}$ such that any $F\in\mathcal{R}$ has infinitely many sinks.

Let $\mathcal{N}_n\subset\mathcal{N}$ be the (open) subset of maps $F\in\mathcal{N}$ with $n$ attracting periodic orbits (at least) disjoint from the critical line $\{x=0\}$ and $\mathcal{R}=\bigcap\limits_{n\in\mathbb{N}}\mathcal{N}_n$. In this notation, our task is reduced to show the next proposition.

\begin{prop}\label{inf-prop}$\mathcal{N}_n$ is dense for every $n\in\mathbb{N}$.
\end{prop}

We start our argument with the following notion.

\begin{defi} We say that $F\in\mathcal{D}^r$ exhibits a ``homoclinic tangency'' if there is a hyperbolic periodic point $p$ of saddle-type of $F$ such that
\begin{enumerate}
 \item there exists $c^{+}=(0^{+},y_p)\in W^u(p)$;
\item there exists $k>0$ such that $F^k(c^+)\in W^s_{loc}(p)$ and $F^j(c^+)$ does not intersect the critical line for $0<j\leq k$;
\end{enumerate}
\end{defi}

The relevance of homoclinic tangencies becomes apparent in the next lemma.

\begin{lema}\label{l.sinks}Let $F\in {\mathcal D}^r$ with a homoclinic tangency. Then, there exists $G\in {\mathcal D}^r$ arbitrarily $C^r$ close to $F$ having an attracting periodic orbit disjoint from the critical line $\{x=0\}$ near the homoclinic tangency.
\end{lema}

\begin{proof}From the fact that $c^{+}\in W^u(p)$, it follows that there exists $c^{+}_n= (0^{+}, y_n)$ such that $F^{-k_n}(c^{+}_n)\to F^k(c^{+})$ (for some appropriate sequence $k_n$) and $c^{+}_n\to c^{+}=(0^+,y_{\infty})$. Indeed, since the unstable manifold of $p$ consists of horizontal segments (as one can check from the form of the maps $F\in\mathcal{D}^r$), we can apply Palis' inclination lemma to the vertical segments $\{0^+\}\times(y_{\infty}-1/n,y_{\infty}+1/n)$ containing $c^+$ to deduce that, for each $n\in\mathbb{N}$, there are $k_n\in\mathbb{N}$ such that $F^{-k_n}(\{0^+\}\times(y_{\infty}-1/n,y_{\infty}+1/n))$ intersect a $(1/n)$-neighborhood of $F^k(c^+)$. Thus, for each $n\in\mathbb{N}$, we can select $c_n^+=(0^+,y_n)\in \{0^+\}\times(y_{\infty}-1/n,y_{\infty}+1/n)$ such that $F^{-k_n}(c_n^+)$ belongs to a $(1/n)$-neighborhood of $F^k(c^+)$. In particular, $c_n^+\to c^+$ and $F^{-k_n}(c_n^+)\to F^k(c^+)$, as desired. 

Hence, we can take a $C^r$ small perturbation $H=H_n$ of $F$ such that $H^k(c_n^{+})=F^{-k_n}(c_n^{+})$ and $H=F$ along the orbit $F^{-j}(c_n^{+})$ for $j=1,\dots,k_n$ (provided that $n$ is large enough). In fact, since $F^k(c^+)$ and the orbit of $p$ don't meet the critical line, we can fix $\alpha>0$ such that, for all $n\in\mathbb{N}$, one has $F^j(c_n^+)\in([-1,-\alpha]\cup[\alpha,1])\times[0,1]$ for $j=1,\dots, k$ and $F^{-m}(c_n^+)\in([-1,-\alpha]\cup[\alpha,1])\times[0,1]$ for $m=1,\dots, k_n+k-1$. In particular, if we take a smooth bump function $\nu_{\alpha}:[-1,1]\to \mathbb{R}$ with $\nu_{\alpha}(x)=0$ if $|x|\geq\alpha$ and $\nu_{\alpha}(x)=1$ for $|x|\leq\alpha/2$, and if we set 
$$H(x,y)=\left(\left(1+(X_n-1)\nu_{\alpha}(x)\right)f(x,y), \left(1+(Y_n-1)\nu_{\alpha}(x)\right)K_{sgn(x)}(y)\right)\in\mathcal{D}^r,$$ 
where $X_n:=\frac{\pi_1(F^{-k_n-k+1}(c_n^+))}{\pi_1(F(c_n^+))}$, $Y_n:=\frac{\pi_2(F^{-k_n-k+1}(c_n^+))}{\pi_2(F(c_n^+))}$, $\pi_1(x,y)=x$ and $\pi_2(x,y)=y$, then we obtain a small $C^r$-perturbation of $F(x,y)=(f(x,y),K_{sgn(x)}(y))$ because $F^{-k_n}(c_n^+)$ is close to $F^{k}(c^+)$ and $c_n^+$ is close to $c^+$, so that $F^{-k_n-k+1}(c_n^+)$ and $F(c_n^+)$ are close to $F(c^+)$ and thus $X_n$ and $Y_n$ are close to $1$. Moreover, by definition, $H(c_n^+)=F^{-k_n-k+1}(c_n^+)$ and  $H=F$ on $\{|x|\geq\alpha\}\times[0,1]$. Since $F^{j}(c_n^+)$ belongs to $\{|x|\geq\alpha\}\times[0,1]$ for $j=-k_n-k+1,\dots, k$, we obtain that $H^k(c_n^{+})=F^{-k_n}(c_n^{+})$ and $H=F$ along the orbit $F^{-j}(c_n^{+})$ for $j=1,\dots,k_n$, and, thus, $c_n^{+}=H^{k+k_n}(c_n^{+})$ is a periodic point of $H$ of period $k+k_n$. Furthermore, using that $c_n^+$ is a critical point of 
$H$, the reader can check that the periodic point $c_n^+$ is attracting. Of course, this does not complete the argument because the attracting periodic point $c_n^+$ of $H=H_n$ belongs to the critical line, but this little inconvenience is easily overcome by slightly perturbing $H=H_n$ as follows. 

Let us write $H=H_n\in\mathcal{D}^r$ as $H(x,y)=(h(x,y), L_{sgn(x)}(y))$. Recall that the $H$-orbit of $c_n^+$ does not intersect the critical line except for $c_n^+=(0^+,y_n)$ itself: indeed, using the notation of the previous paragraph, $H^i(c_n^+)\in\{|x|\geq\alpha\}\times[0,1]$ for all $i=1,\dots, k+k_n-1$. Thus, if we denote by $\theta_l=\alpha/2l$, we see that, for all $l$ sufficiently large, we can slightly perturb the unimodal maps $h(x,y_n)$ near $x=0$ and $h(x,\pi_2(H^{k+k_n-1}(c_n^+)))$ near $x=\pi_1(H^{k+k_n-1}(c_n^+))$ such that the resulting map $G=G_l$ satisfy $G(\theta_l,y_n)=H(c_n^+)$, $G(H^{k+k_n-1}(y))=(\theta_l,y_n)$ and $G=H$ nearby the points $H^i(c_n^n)$, $i=1,\dots,k+k_n-2$. In particular, the resulting map $G=G_l\in\mathcal{D}^r$ has an attracting periodic point of period $k+k_n$ at $(\theta_l,y_n)$ whose $G$-orbit does not meet the critical line. This ends the proof.
\end{proof}

On the other hand, homoclinic tangencies are frequent inside $\mathcal{N}$.
\begin{lema}\label{l.hetero}Let $F\in {\mathcal N}$. Then, there exists $G$ $C^r$ close to $F$ exhibiting a homooclinic tangency.
\end{lema}

\begin{proof}This is an immediate consequence of the construction of $\mathcal{N}$: given $F\in\mathcal{N}$, we can find $x_1, x_2\in\La_{\e_m}$ and a critical point $c^{+}\in W^u_{loc}(x_1)$ such that $F^k(c^{+})\in W^s_{loc}(x_2)$. Since the hyperbolic set $\Lambda_{\e_m}$ is transitive, we can find a periodic $p$ in $\La_{\e_m}$ close to $x_1$ so that some pieces of its local unstable and stable manifolds are close to the corresponding invariant manifolds of $x_1$ and $x_2$ (resp.). In this situation, after a proper small perturbation (similar to the ones performed during the proof of Lemma \ref{l.sinks} above), we can find $G$ $C^r$-close to $F$ such that $G\in\mathcal{D}^r$ has a homoclinic tangency involving $p$.
\end{proof}

Finally, the proof of the desired proposition follows from a direct combination of the two previous lemmas.

\vskip 3pt

\noi{\em Proof of Proposition \ref{inf-prop}.} It is proved by induction. Given $G_n\in {\mathcal N}_n$, we can use Lemma~\ref{l.hetero} to find $G_{n+1}$ $C^r$-close to $G_n$ keeping the same number $n$ of attracting periodic points of $G_n$ and such that $G_{n+1}$ has a homoclinic tangency. By Lemma~\ref{l.sinks}, we can unfold this tangency to create a new sink, i.e., we can find $H\in\mathcal{N}_{n+1}$ $C^r$-close to $G_{n+1}$. The result follows by Baire's theorem.
\qed

This completes the proof of Theorem A.


\section{ Proof of Theorem B}\label{s.B}

\vskip 3pt

Before giving the proof of Theorem B, we briefly outline the strategy. Given $\e>0$, let us take $U_\e=([-1,-\e]\cup [\e,1])\times [0,1]$
and
$$\La_\e=\Omega(F)\cap\bigcap\limits_{n \in \ZZ} F^n(U_\e).$$

\noi {\bf Strategy of the proof.}

\bigskip

\begin{enumerate}
\item For any $\e>0$, we show that, $C^1$-generically, the set
$\La_\e$ is composed by a locally maximal hyperbolic set and a finite number of
periodic attracting points. This is performed in Subsection
\ref{subhyp}  (see Theorem \ref{hyp} and Corollary \ref{c.loc-max}).
\item We show that, $C^1$-generically, any critical point either it is contained
in the basin of attraction of the sinks (of Step 1 above) or
it returns to $[-\e,\e]\times [0,1]$. This is performed in Subsection \ref{subret}.
\item Later, we produce a series of $C^1-$perturbations (of size proportional to $\e$) in the way to
create a finite number of periodic sinks such that their basins
contain the critical points that return. This is performed in
Subsection \ref{subpert}.
\item From items  $1$, $2$ and $3$, it follows that
$\Omega(F)\subset \La_\e\cup\{p_1,....,p_k\}$, where each $p_i$
is a periodic attracting point ($i=1,\dots,k$), and therefore it is
concluded that $\Omega(F)$ is hyperbolic (and $F$ is Axiom A).
\end{enumerate}

\bigskip

\noi {\bf Convention.} As it will become clear later, although our Theorem B concerns $\mathcal{D}^1$ with the $C^1$-topology, sometimes we will need some extra ($C^2$) regularity of the maps (for instance, this is the case when one tries to apply a 
fundamental ``hyperbolicity criterion for surface maps'' of Pujals and Sambarino). 
Therefore, during this \emph{entire} (last) Section~\ref{s.B},
we consider $\mathcal{D}^2$ equipped with the $C^1$-topology. In particular, each time we refer to ``a $C^1$-generic/typical $F\in\mathcal{D}^2$'', we mean a $F\in\mathcal{D}^2$ belonging to an appropriate $C^1$-open and $C^1$-dense subset of $\mathcal{D}^2$.

\subsection{Hyperbolicity of $\La_\e$}\label{subhyp}

\begin{teo}\label{hyp}
Let $\e>0$ be a positive constant. Then, there exists a $C^1$-open and dense subset $\mathcal{V}_{\e}\subset\mathcal{D}^1$ such that, for any $F\in\mathcal{V}_{\e}$, the set $\La_\e$ contains a finite number
of periodic attracting points and the complement of the basin of
attraction of them $\hat \La_\e$ exhibits a hyperbolic splitting
$T\hat \La_\e=E^s\oplus E^u$ such that $E^s$ is contracting, $E^u$ is
expanding (and, in fact, $E^u=\mathbb{R}\cdot(1,0)$).
\end{teo}

The proof of this result uses the notion of dominated splitting and Theorem B in \cite{PS1}. Firstly, we revisit the definition of dominated splittings:

\begin{defi}A $f$-invariant set
$\Lambda$ has a \emph{dominated splitting} if we can decompose
its tangent bundle into two invariant subbundles $T_\Lambda M=E\oplus
F$ such that:
\begin{eqnarray}\label{spdom}\DD{n}{x}{}\le C\lambda^n, \mbox{ for all } x\in\Lambda, n\ge 0.\end{eqnarray}
\noi with $C>0$ and $0<\lambda<1.$
\end{defi}

Secondly, we recall that Pujals and Sambarino~\cite{PS1} proved that any compact invariant set
exhibiting dominated splitting of a generic $C^2$ surface diffeomorphism is hyperbolic:

\begin{teo}[\cite{PS1}]\label{teops1}
Let $f\in \textrm{Diff}^2(M^2)$ be a $C^2$-diffeomorphism of a compact surface $M^2$ and $\La\subset\Omega(f)$ a compact invariant set exhibiting a dominated splitting. Assume that all periodic points in $\Lambda$ are hyperbolic of saddle type. Then, $\La$ can be decomposed into a hyperbolic set and a finite number of normally hyperbolic periodic closed curves whose dynamical behaviors are $C^2$-conjugated to irrational rotations.
\end{teo}




Let us begin the proof of Theorem \ref{hyp} with some useful notation. Given $(x_0,y_0)$, we denote by $(x_i, y_i):= F^i(x_0, y_0)$; also, we write the derivative of a map $F(x,y)=(f(x,y), K(x,y))$ of the form~(\ref{e.F}) as
$$DF = \left( \begin{array}{cc}
f_x & f_y \\
0 & K_y \end{array} \right).$$
In particular, it follows that
$$DF^n(x_0,y_0)= \left( \begin{array}{cc}
A_n & B_n \\
0 & D_n \end{array} \right),$$
where $$A_n:= \Pi_{i=0}^{n-1} f_x(x_i,y_i)
\,\,\,\, , D_n:= \Pi_{i=0}^{n-1} K_y(x_i,y_i),$$ $$B_n=
\sum_{j=0}^{n-1} f_y(x_j,y_j)\,\Pi_{i=0}^{j-1} K_y(x_i,y_i)\,
\Pi_{i=j+1}^{n-1} f_x(x_i,y_i).$$
In the sequel, we fix two positive constants $\la_0, \la_1$ such that $\la_0<\la_1 <1$ and $$|K_y|<\la_0.$$
Concerning the proof of Theorem \ref{hyp}, we observe that $(1,0)$ is an invariant direction by $DF$ and, moreover, it is the natural candidate to be the expanding one. Therefore, the existence of a dominated
splitting follows once we build up a proper invariant cone field around $(1,0)$. To perform this task, first we need the next lemma.

\begin{lema} \label{pocos}
Given $\e>0$, let $\mathcal{Z}_{\e}\subset\mathcal{D}^1$ be the $C^1$-open and dense subset consisting of $F\in\mathcal{D}^1$ such that, for each $y\in[0,1]$, the unimodal map $f(.,y)$ has no critical points in $|x|\geq\e$. Then, for any $F\in\mathcal{Z}_{\e}$, there exist a finite number of attracting periodic points with
trajectory in $\La_\e$ and a positive integer $n_0=n_0(\varepsilon)$ such that, for
any $(x_0, y_0)\in \La_\e$ outside the basins of
attraction of those periodic points, it holds
$$ |A_n|= \Pi_{i=0}^{n-1} \,|f_x(x_i,y_i)| > \la_1^n,$$
whenever $n>n_0$.
\end{lema}

In order to do not interrupt the flow of ideas, we postpone the proof of the lemma. Assuming momentarily this lemma, we are able to prove Theorem \ref{hyp}.

\vskip 1pt

\noi{\bf Proof of Theorem \ref{hyp}.}
Let $b$ be a positive constant such that $$|f_y| < b.$$ Given $F\in\mathcal{Z}_{\e}$, we can take $n_0$ the integer provided by Lemma \ref{pocos}
and let $R_0$ be a positive constant\footnote{Such a constant $R_0$ always exists since $f(x,y)\in \mathcal{U}^1$ is a family of unimodal maps without critical points in 
$|x|\geq\e$ for $F\in\mathcal{Z}_{\e}$ and $(x_i,y_i)\in U_\e$ implies $|x_i|\geq\e$.} such that, for any $m < n_0$ and any point $(x_0,y_0)\in\La_\e$, it holds
$$ \Pi_{i=0}^m |f_x(x_i,y_i)| > R_0^{-1}.$$
Now, for all $(x_0,y_0)\in\La_\e$ outside the basins of the attracting periodic points of Lemma~\ref{pocos}, let us bound $B_n$ for $n>n_0$:
\begin{eqnarray}\label{estima}
|B_n|&\leq& \sum_{j=0}^{n-1} |f_y(x_j,y_j)|\Pi_{i=0}^{j-1} |K_y(x_i,y_i)| \,
\Pi_{i=j+1}^{n-1} |f_x(x_i,y_i)|\\ \nonumber&=& \sum_{j=0}^{n-1} |f_y(x_j,y_j)|\cdot |D_j|\cdot\frac{|A_n|}{|A_{j+1}|}\\ \nonumber
&<& R_0 b |A_n|\frac{1}{1-\la_0} + b |A_n|\sum_{j=n_0}^{n-1}  \frac{\la_0^{j}}{\la_1^j} \\ \nonumber
&<&  R_0 b |A_n|\frac{1}{1-\la_0} + b |A_n|
\frac{1}{\lambda_1-\lambda_0}.
\end{eqnarray}
Using this estimate, we claim that the cone field 
$C(\ga_0):=C(\mathbb{R}\cdot(1,0),\ga_0):=\{(\dot{x},\dot{y})\in\mathbb{R}^2:|\dot{y}|\leq\ga_0|\dot{x}|\}$
is a forward invariant cone field for sufficiently small $\ga_0>0$. In fact, take $\ga_0>0$ small and let us consider $$v_n= DF^n(1,\ga)=(A_n+ \ga B_n, \ga D_n),$$
where $|\ga|<\ga_0$. The slope of $v_n$ with respect to $(1,0)$ is
$$|\textrm{slope}(v_n, (1,0))|= \frac{|\ga D_n|}{|A_n+ \ga B_n|}.$$
Note that the estimate~(\ref{estima}) implies
$$|A_n + \ga B_n| > |A_n| - \ga_0|B_n| > |A_n|\left(1- \ga_0 b(R_0 + 1)\cdot(\la_1-\la_0)^{-1} \right).$$
Hence, if $\ga_0$ is small so that
$$1- \ga_0 b(R_0 +1)(\la_1-\la_0)^{-1} > \frac{1}{2},$$
using Lemma~\ref{pocos}, we conclude that
$$\textrm{slope}(v_n, (1,0)) < \ga_0 \frac{2|D_n|}{|A_n|} < 2\left(\frac{\la_0}{\la_1}\right)^n\cdot\ga_0.$$
Thus, assuming $n_0$ large so that $(\la_0/\la_1)^{n_0}<1/4$ and taking
$\ga_1=  2(\la_0/\la_1)^{n_0}\ga_0$, we see that, for any $n>n_0$,
$$DF^{n}(C(\ga_0))\subset C(\ga_1)\subset C(\ga_0/2).$$ In other words, $C(\ga_0)$ is a forward invariant cone field and the existence of a dominated splitting $E^s\oplus\mathbb{R}\cdot (1,0)$ is guaranteed (over the set $\hat \La_\e$ of points outside the basins of the attracting points of Lemma~\ref{pocos}).

Next, we show that $E^s$ is uniformly contracted: for every $(x_0,y_0)\in\La_\e$, we fix $e_0^{(s)}=(u_0^s,v_0^s)\in E_{(x_0,y_0)}^s$ with $\|e_0^{(s)}\|=1$ and we put $DF^n(x_0,y_0)\cdot e_0^{(s)} := \pm\la_n^s\cdot e_n^{(s)}\in E_{(x_n,y_n)}^s$ where $e_n^{(s)}:=(u_n^{(s)},v_n^{(s)})\in E_{(x_n,y_n)}^s$ is an unitary vector. Then, we compute the determinant of $DF^n$:
$$|A_n\cdot D_n| = |\det DF^n| = \frac{|DF^n\cdot(1,0) \wedge DF^n\cdot e_0^{(s)}|}{|(1,0)\wedge e_0^{(s)}|}
= \frac{|A_n|\cdot |\lambda_n^{(s)}|\cdot |v_n^{(s)}|}{|v_0^{(s)}|},$$
where $|u\wedge v|$ denotes the area of the rectangle determined by the vectors $u$ and $v$. Because the direction $E^s$ does not belong to the cone field $C(\ga_0)$ and $|v_0^{(s)}|\leq \|e_0^{(s)}\|=1$, we get
$$|\lambda_n^{(s)}|=|D_n|\frac{|v_0^{(s)}|}{|v_n^{(s)}|}\leq \frac{1}{\ga_0}|D_n|.$$
Since $|D_n|\leq \lambda_0^n$ for all $n\in\mathbb{N}$, this proves that for all $F\in\mathcal{Z}\subset\mathcal{D}^1$ of the form~(\ref{e.F}) such that $|K_y|<\lambda_0$ and for any $\lambda_0<\lambda_1<1$, the set $\La_\e$ is the union of a finite number of sinks and a set $\hat\La_\e$ exhibiting a dominated decomposition $E^s\oplus F$ where $E^s$ is contracting (after $n_0$ iterates) and $F=(1,0)\cdot\mathbb{R}$ satisfies $DF^n(1,0)=(A_n,0)$ where $|A_n|>\lambda_1^n$ (for $n>n_0$). 

At this stage, the proof of Theorem \ref{hyp} is reduced to show that there exists a $C^1$-dense subset of $F\in\mathcal{D}^2\cap\mathcal{Z}_{\e}$ such that all periodic points in $\hat\La_\e$ are hyperbolic of saddle type. Indeed, this is true because it is immediate that there are no periodic closed curves inside $\La_\e$ whose dynamical behavior are conjugated to irrational rotations\footnote{Because in this case $F=\mathbb{R}\cdot(1,0)$ should be a tangent line of such closed curve $C$ at some point. Combining this fact with the minimality of the dynamics on $C$ and the continuity of dominated splitting (besides the invariance of $\mathbb{R}\cdot(1,0)$), we obtain that the whole curve $C$ is tangent to the line field $F$, a contradiction.}, so that, by Pujals-Sambarino Theorem \ref{teops1}, a dominated splitting over $\hat\La_\e$ is a hyperbolic splitting whenever all periodic points in $\hat\La_\e$ are hyperbolic of saddle type.

However, the $C^1$-denseness in $\mathcal{D}^1$ of $F\in\mathcal{D}^2\cap\mathcal{Z}_{\e}$ such that all periodic points in $\hat\La_\e$ are hyperbolic of saddle type is a consequence of a simple argument (compare with \cite[p.~966]{PS1}): recall that, by the ``usual'' transversality arguments, all periodic points of a $F\in\mathcal{D}^2\cap\mathcal{Z}_{\e}$ in a $C^2$-generic ($G_{\delta}$ dense) subset are hyperbolic\footnote{In our context of $F$ of the form (\ref{e.F}), the derivative $DF$ is an upper triangular matrix whose diagonal entries are the $x$-derivative of $f(x,y)$ and the $y$-derivative of $K_{\pm}(y)$. In particular, given any non-hyperbolic periodic point of $F$, we can slightly perturb 
$K_{\pm}$ to slightly change both the trace and the determinant of $DF$ to get an hyperbolic periodic point. Therefore, the set of $F\in\mathcal{D}^2$ whose periodic points of period $\leq n$ are hyperbolic are $C^2$-open and dense, and, hence, the set of $F\in\mathcal{D}^1$ whose periodic points are all hyperbolic form a $C^2$ $G_{\delta}$-dense.}\label{footnote-KS-1}; it follows that for such a $F\in\mathcal{D}^2\cap\mathcal{Z}_{\e}$, the compact invariant subset $\La_\e^{(0)}:=\La_\e-\{p\in\La_\e: p \textrm{ is a periodic sink}\}\subset\Omega(F)$ only contains hyperbolic periodic points of saddle type. Furthermore, $\La_\e^{(0)}$ admits a dominated splitting (since $\La_\e^{(0)}\subset\hat\La_\e$). Thus, we obtain from Theorem~\ref{teops1} that $\La_\e^{(0)}$ is a hyperbolic set. We claim that $P_\e(F):=\La_\e-\La_\e^{(0)}$ is finite (so that $\La_\e^{(0)}=\hat\La_\e$ and, \emph{a fortiori}, all periodic points of $\hat\La_\e$ are hyperbolic of saddle-type). Indeed, if $\# P_\e(F)=\infty$,     we have $\emptyset\neq \overline{P_\e(F)}-P_\e(F)\subset\La_\e^{(0)}$. However, since $\La_\e^{(0)}$ is hyperbolic, we can select a compact neighborhood $U$ of $\La_\e$ such that the maximal invariant of $U$ is hyperbolic. Thus, we get that, up to removing a finite number of periodic sinks, $P_\e(F)\subset U$, a contradiction with the hyperbolicity of the maximal invariant subset of $U$. This completes the proof of Theorem \ref{hyp}.\qed

Closing the proof of the hyperbolicity of $\La_\e$, we prove the statement of Lemma~\ref{pocos}.

\vskip 5pt

\noi{\bf Proof of Lemma \ref{pocos}.} It is enough to apply the
following lemma due to Pliss (see \cite{Pl}, \cite{M2}).

\begin{lema}[Pliss]
\label{pliss}
Given $0< \gamma_0 < \gamma_1 < 1 $ and $a>0$, there exist
$n_0=n_0(\gamma_0,\gamma_1,a)$ and $l=l(\gamma_0,\gamma_1,a)>0$ such
that, for any sequences of numbers $\{a_i\}_{0\leq i \leq n}$ with
$n_0<n$, $ a^{-1} < a_i < a$ and $\Pi_{i=0}^n  a_i< \gamma_0^{n}$, there are $1\leq n_1<n_2<\dots<n_r\leq n$ with
$r>ln$ and such that
$$  \Pi_{i=n_j}^k a_i < \gamma_1^{k-n_j} \,\,\,\,\, \textrm{ for all } \,\,\, n_j\leq k\leq n.$$
\end{lema}

In fact, let us consider the set of points $(z,w)\in \La_\e$ such that
\begin{eqnarray}\label{non}
\liminf\limits_{n\to\infty} \frac{1}{n} \log |A_n(z,w)| < \log\sqrt{\la_1}.
\end{eqnarray}
Since $F\in\mathcal{Z}_{\e}$ and for any $(z_i,w_i)=F^i(z,w)\in\La_\e$, it holds $|z_i|\geq\e$, we have that the numbers $a_i=f_x(z_i,w_i)$ are uniformly bounded away from zero and infinity (i.e., $0<a(\e)^{-1}<a_i<a(\e)<\infty$), and, thus, we can use Lemma
\ref{pliss} twice to obtain that there exists a subsequence of forward iterates of $(z,w)$ accumulating on some point $(x_0,y_0)$ which has a subsequence of forward iterates
$$\{(x_{n_j},y_{n_j})\}_{j>0}=\{F^{n_j}(x_0,y_0)\}$$
such that any $(x_{n_j},y_{n_j})$ satisfies 
$$|A_n(x_{n_j},y_{n_j})| < \sqrt{\la_1^n},\,\,\, \forall \,\, n>0.$$
Using the same type of calculation of estimative~(\ref{estima}), we get, for any $j>0$,
$$\Pi_{i=0}^n ||DF(x_{i+n_j},y_{i+n_j})||< (1+b(\sqrt{\la_1}-\la_0)^{-1})\sqrt{\la_1^n},\,\,\, \forall \,\, n>0 \,\,  \textrm{large}.$$
By standard arguments it follows that, for any $\sqrt{\la_1}<\la_2<1$, there exists $\ga=\ga(\la_1,\la_2)$ such that
$$F^n(B_\ga(x_{n_j},y_{n_j}))\subset
B_{\la_2^n\ga}(F^n(x_{n_j},y_{n_j}))$$ for all $j,n>0$ large. Taking $q_0$ an accumulation point of $\{(x_{n_j},y_{n_j})\}$, it is not hard to see that
$$F^j(B_{\ga}(q_0))\subset
B_{\la^j_2\ga}(F^j(q_0))$$
for any $j>0$ large and there exists a positive integer $m=m(q_0)$ such that
$$F^m(B_{\ga}(q_0))\subset
B_{\la^m_2\ga}(q_0).$$
Therefore, it follows that:
\begin{enumerate}
\item there is an unique attracting periodic point\footnote{Actually, using that
$(x_{n_j},y_{n_j})=F^{n_j}(x_0,y_0)\to q_0$, it can be concluded that $q_0$
is the periodic point.} $p_0$ inside $B_{\ga}(q_0)$,
\item  the neighborhood  $B_{\ga}(q_0)$ is contained in the basin of attraction of $p_0$,
\item the point $(x_0,y_0)$ and the initial point $(z,w)$ verifying (\ref{non}) belong to the basin of attraction $p_0$;
\end{enumerate}
Since the number of attracting periodic point with local basin of attraction with radius
larger than $\ga$ is finite, we conclude that there are
a finite number of periodic attracting points whose basins contain the points of $\La_\e$ verifying (\ref{non}).
In other words, $\La_{\e}$ is the union of finitely many periodic sinks and the subset $\hat \La_{\e}$ of $\La_{\e}$ consisting of the points violating (\ref{non}). Note that $\hat\La_\e$ is compact because it lies in the complement of the basins of attraction of finitely many sinks. At this point, the proof of the lemma will be complete if we show that there exists $n_0=n_0(\e)\in\mathbb{N}$ such that $|A_n(z,w)|>\lambda_1^n$ for all $n>n_0$ and $(z,w)\in\hat\La_\e$. However, the existence of such an integer $n_0$ follows from Pliss lemma and the arguments of the previous paragraphs. Indeed, the non-existence of $n_0$ would imply that there exists a sequence $(z_i,w_i)\in\hat\La_\e$ and $n_i\in\mathbb{N}$, $n_i\to\infty$, such that $|A_{n_i}(z_i,w_i)|\leq\lambda_1^{n_i}$. Using Pliss' lemma (as in the previous paragraphs), we would be able to extract a subsequence of $(z_i,w_i)$ accumulating in some point $(z_{\infty}, w_{\infty})$ verifying (\ref{non}). Of course, this is a contradiction because $\hat\La_\e$ is compact (and thus $(z_{\infty},w_{\infty})\in\hat\La_\e$ and this concludes the proof of the lemma. \qed

For later use, we observe that the hyperbolic sets $\La_\e$ can be assumed to be locally maximal. This follows from the next claim (compare with \cite{A}).

\begin{afir}\label{basic} There exists a locally maximal hyperbolic set $\La_\e\subset \widetilde \La_\e\subset U_{\e/2}$. In other words, there exists and open set such $V\subset U_{\e/2}$ containing $\La_\e$ such that $\widetilde \La_\e:=\cap_{n\in \Z} F^n(V)$ is a compact invariant hyperbolic set.
 \end{afir}

\begin{proof}  Indeed, fix $\gamma=\gamma(\e)>0$ a positive small constant such that the local stable manifold $W_{\gamma}^s(p)$ of any point $p\in\La_{\e/2}$ is the graph of a real function of the $y$-coordinate defined over an interval of length $\delta=\delta(\e)>0$. Next, we take $k=k(\e)>0$ a large integer so that the lengths of the $2^k$ intervals $I_1^{(k)},\dots,I_{2^k}^{(k)}$ of the $k$th stage of the construction of the Cantor set $\mathcal{K}_0$ are $<\delta/2$. Note that we can suppose that $W_{\gamma}^s(p)\subset U_{\e/2}$ for any $p\in\La_{\e/2}\cap U_{3\e/4}$. Now, for each $j=1,\dots,2^k$, we consider the stable lamination $\mathcal{F}_{j,\pm}^s = \{W_\gamma^s(p)\cap [-1,1]\times I_j^{(k)}\}_{p\in\hat\La_{\e/2}\cap U_{3\e/4}}$. Given $\ell\in\mathcal{F}_{j,-}^s$, resp. $\ell\in\mathcal{F}_{j,+}^s$, we denote by $R_{j,-}^{(k)}(\ell)$, resp. $R_{j,+}^{(k)}(\ell)$, the rectangle delimited by the four lines $\{-1\}\times[0,1]$, $[-1,1]\times\partial I_j^{(k)}$ and $\ell$, resp. $\{+1\}\times[0,1]$, $[-1,1]\times\partial I_j^{(k)}$ and $\ell$. Given $\ell,\widetilde{\ell}\in\mathcal{F}_{j,\pm}^s$, we say that $\ell\prec\widetilde{\ell}$ if and only if $R_{j,\pm}^{(k)}(\ell)\subset R_{j,\pm}^{(k)}(\widetilde{\ell})$. Observe that $\prec$ is a \emph{total order}\footnote{Because any two distinct stable leaves are disjoint and $\partial\ell\subset[-1,1]\times\partial I_j^{(k)}$ for any $\ell\in\mathcal{F}_{j,\pm}^s$.} of $\mathcal{F}_{j,\pm}^s$. Thus, for each $j=1,\dots,2^k$, we can define $\ell_{j,\pm}\in\mathcal{F}_{j,\pm}^s$ the \emph{outermost} stable leaf of $\hat\La_{\e/2}\cap U_{3\e/4}\cap[-1,1]\times I_j^{(k)}$ as the unique leaf of $\mathcal{F}_{j,\pm}^s$ such that $\ell\prec\ell_{j,\pm}$ for all $\ell\in\mathcal{F}_{j,\pm}^s$. Consider the family of rectangles $R_{j,\pm}^{(k)}(\e):= R_{j,\pm}^{(k)}(\ell_{j,\pm})$. Finally, let $\widetilde\La_\e$ be the maximal invariant set associated to this family of rectangles. It follows that $\widetilde\La_\e$ has local product structure (because $W_{loc}^s(p)\cap W_{loc}^u(q)\in R_{j,\pm}^{(k)}(\e)$ when $p,q\in\widetilde\La_\e \cap R_{j,\pm}^{(k)}(\e)$) and $\La_\e\subset\widetilde\La_\e\subset U_{\e/2}$ (because $\La_\e\cap [-1,1]\times I_j^{(k)}\subset R_{j,\pm}^{(k)}\subset U_{\e/2}$). This proves our claim.
 \end{proof}

\begin{afir}\label{basicII} The set $\widetilde\La_\e$ found in Claim~\ref{basic} is the maximal invariant set of $U_\e\cup\widetilde{R_\e}$, where $\widetilde{R_\e}=\{R_{j,\pm}^{(k)}(\ell_{j,\pm})\}$ is the family of rectangles introduced in the previous claim. 
 \end{afir}
\begin{proof}
Indeed, given $z$ a point whose orbit $\mathcal{O}(z)$ stays in $U_\e\cup\widetilde{R_\e}$, we note that $z\in\Lambda_{\e/2}$ (since $U_\e\cup\widetilde{R_\e}\subset U_{\e/2}$). On the other hand, we have two possibilities:
\begin{itemize}
\item $\mathcal{O}(z)\subset\widetilde{R_\e}$: this means that $z\in\widetilde\Lambda_\e$;
\item there exists $y\in\mathcal{O}(z)-\widetilde{R_\e}$: this means that $y\in (U_\e\cap\Lambda_{\e/2})-\widetilde{R_\e}$, a contradiction (since, by definition, $U_\e\cap\Lambda_{\e/2}\subset U_{3\e/4}\cap\Lambda_{\e/2}\subset\widetilde{R_\e}$).
\end{itemize}
In particular, it follows that the positive orbit $\mathcal{O}^+(p)$ of every point $p\notin W^s(\widetilde{\La_\e})$ escapes any sufficiently small \emph{neighborhood} of $U_\e\cup\widetilde{R_\e}$. In fact, if the positive orbit of a given point $p$ stays forever inside a small neighborhood $W$ of $U_\e\cup\widetilde{R_\e}$, its accumulation points always belong to the maximal invariant set $\Lambda(W)$ of $W$. However, since the maximal invariant $\widetilde{\La_\e}$ of $U_\e\cup\widetilde{R_\e}$ is locally maximal (by Claim~\ref{basic}), $\Lambda(W)=\widetilde{\La_\e}$ for any small neighborhood $W$ of $U_\e\cup\widetilde{R_\e}$. Hence, $p\in W^s(\widetilde{\La_\e})$, a contradiction.
\end{proof}

Before proceeding further, we use a fundamental result of C. G. Moreira to improve the geometry of the isolating neighborhood of $\widetilde{\La_\e}$.

\begin{teo}[\cite{M}]\label{gugu} Let $\mathcal{K}$ be a $C^2$-dynamically defined Cantor set and let $\widetilde{\mathcal{K}}$ be a $C^1$-dynamically defined Cantor set. Then, there are $C^1$-dynamically defined Cantor sets $\widehat{\mathcal{K}}$ arbitrarily $C^1$-close to $\widetilde{\mathcal{K}}$ such that $\mathcal{K}\cap\widehat{\mathcal{K}}=\emptyset$. In particular, generically in the $C^1$-topology,
a pair of $C^1-$dynamically defined Cantor sets are disjoint and the arithmetic difference of a $C^1$ generic pair of $C^1-$dynamically defined Cantor sets has empty interior (so that it is also a Cantor set).
\end{teo}

More precisely, combining our Theorem \ref{hyp} with this theorem, we have the following consequence:
\begin{cor}\label{c.loc-max}Fix $\e>0$. Then, for $F$ in a $C^1$-open and dense subset of $\mathcal{D}^1$, the maximal invariant set $\La_\e$ of $U_\e$ is a locally maximal hyperbolic set such that $\textrm{int}(U_\e)$ is an isolating neighborhood of $\La_\e$.
\end{cor}

\begin{proof}Let $F\in\mathcal{V}_{\e/2}\cap\mathcal{Z}_{\e}\cap\mathcal{D}^2$ where $\mathcal{V}_{\e/2}$ is the $C^1$-open and dense subset of $\mathcal{D}^1$ verifying Theorem~\ref{hyp}. We consider a finite Markov partition $\{P_i\}_{i=1}^M$ of $\widetilde{\La_{\e/2}}$ with small diameter. We take $p_i\in P_i\cap\widetilde{\La_{\e/2}}$ and we define $E_i:=E^s(p_i)$. Since the stable foliation $W^s_i(x)=W^s_{loc}(x)\cap P_i$ of $F$ restricted to $x\in P_i$ is $C^1$-close to the foliation of $P_i$ by straight lines with direction $E_i$ when the diameter of the Markov partition is small, we can assume, up to performing a $C^{1}$-perturbation of the unimodal family $f(x,y)$, that the stable foliation of $F$ restricted to $P_i$ is the foliation by straight lines parallel to $E_i$. Indeed, if we consider the horizontal segments $\gamma^u_i$ forming the bottom unstable boundary of $P_i$ and we denote by $\{\Pi_i(x,y)\}=((x,y)+\mathbb{R}E_i)\cap\gamma_i^u$, we can define $g(x,y)$ as the $x$-coordinate of the point $\{(g(x,y), K_{sgn(x)}(y))\}=(F(\Pi_i(x,y))+\mathbb{R}E_j)\cap(\mathbb{R}\times\{K_{sgn(x)}(y)\})$ when $F(\Pi_i(x,y))\in P_j$ and a map $G(x,y)=(g(x,y), K_{sgn(x)}(y))$. Geometrically, $G$ sends the point $\{(x,y)\}=(\Pi_i(x,y)+\mathbb{R}E_i)\cap (\mathbb{R}\times\{y\})$ to $(F(\Pi_i(x,y))+\mathbb{R}E_j)\cap(\mathbb{R}\times\{K_{sgn(x)}(y)\})$ and $F$ sends $\{(\overline{x},y)\}=W_i^s(\Pi_i(x,y))\cap(\mathbb{R}\times\{y\})$ to $W_i^s(F(\Pi_i(x,y)))\cap(\mathbb{R}\times\{K_{sgn(x)}(y)\})$. Since the stable foliation $W_i^s$ is $C^1$ to the linear foliation $((x,y)+\mathbb{R}E_i)\cap P_i$, the map $G$ is $C^1$-close to $F$ and thus $G\in\mathcal{D}^1$ (i.e., $g(x,y)$ is unimodal in the $x$-variable) as $F\in\mathcal{Z}_{\e}$. So, from now on, let us suppose that the stable foliation of $F$ restricted to $P_i$ is the foliation by straight lines parallel to $E_i$.

Recall that the angle between the stable directions $E_i$ and the unstable (horizontal) directions is uniformly bounded away from zero. In particular, we also have a system of coordinates on each $P_i$ (given by the horizontal foliation and the foliation by lines parallel to $E_i$) where we can write $F|_{P_i}(x,y)=(f_i(x),K_{sgn(x)}(y))$ and $\widetilde{\La_{\e/2}}\cap P_i$ is a product of two dynamically defined Cantor sets, i.e., $\widetilde{\La_{\e/2}}\cap P_i = \mathcal{K}_i^s\cdot (1,0) + \mathcal{K}_i^u\cdot (\mu_i,1)$ with $\mathcal{K}_i^s$, $\mathcal{K}^u_i$ dynamical Cantor sets of the real line and $(\mu_i,1)\in E_i$.

In this context, the fact that the verticals $\{\pm\e\}\times [0,1]$ don't intersect $\widetilde{\La_{\e/2}}$ is equivalent to $\pm\e\notin \mathcal{K}_i^s+\mu_i\cdot \mathcal{K}_i^u$ for every $i=1,\dots, M$. However, this property can be achieved by a $C^1$-typical perturbation $\widehat{F}$ of $F\in\mathcal{D}^2$: by Moreira's theorem~\ref{gugu}, we can choose, for each $i$, a $(\hat{\mathcal{K}_i^s},\hat{f_i})$ $C^1$-dynamically defined Cantor set $C^1$-close to $(\mathcal{K}_i^s,f_i)$ so that $\pm\e\notin \hat{\mathcal{K}_i^s}+\mu_i\cdot \mathcal{K}_i^u$, and, consequently, $\widehat{F}|_{P_i}(x,y):=(\hat{f_i}(x), K_{sgn(x)}(y))\in\mathcal{D}^2$ has the desired property.
\end{proof}

\subsection{(Quasi) Critical points eventually return}\label{subret}
\begin{defi}Given $\e>0$, we call any point $(\pm\e,y)$ with $y\in \mathcal{K}_0$ a $\e$-quasi-critical point (or simply quasi-critical point).
\end{defi}

Now we use again C. G. Moreira's fundamental result (Theorem~\ref{gugu}) to show that,
for a $C^1$ generic $F\in\mathcal{D}^2$, any quasi critical point returns to the
``critical region''. In other words, roughly speaking, the next result states that
we can avoid in the $C^1$ topology the thickness obstruction (responsible for $C^2$ Newhouse phenomena).
\begin{lema}\label{return}Let $\e>0$ be a positive constant. Then, for $F$ in a $C^1$-generic ($G_{\delta}$-dense) subset of $\mathcal{D}^1$, there exists $m_0\in\mathbb{N}$ such that any quasi-critical point $(\pm\e,y)\in\{\pm\e\}\times \mathcal{K}_0$ satisfies
$$F^{m_y}(\pm\e,y) \in R_\e:=(-\e,\e)\times [0,1]\,\, and \,\, F^{m}(\pm\e,y) \notin (-\e,\e)\times [0,1],\,\, \forall \,\, 0< m < m_y$$
for some positive integer $m_y \leq m_0$ or it is contained in the basin of attraction of some of the (finitely many) attracting periodic points of Theorem \ref{hyp}.
\end{lema}

\begin{proof}Take $F\in\mathcal{D}^1$ with the properties described during the proof of Corollary \ref{c.loc-max}. Since the maximal invariant set $\La_\e$ of $U_\e$ is the union of a finite number of periodic sinks and a hyperbolic set $\hat{\La}_\e$ of saddle type, we see that our task is equivalent to show that
$$\left(\bigcup\limits_{k\geq 0}F^k(\{\pm\e\}\times \mathcal{K}_0)\right)\bigcap W^s_{loc}(\hat{\La}_\e)=\emptyset$$
for a $C^1$-generic $F\in\mathcal{D}^1$. Keeping this goal in mind, given $N\in\mathbb{N}$, we define
$$\mathcal{G}_N:=\{F\in\mathcal{D}^1: \left(\bigcup\limits_{k=0}^{N}F^k(\{\pm\e\}\times \mathcal{K}_0)\right)\bigcap W^s_{loc}(\hat{\La}_\e)=\emptyset\}.$$
It follows that the proof of the lemma is complete once we show that $\mathcal{G}_N$ is $C^1$-dense (because it is clearly $C^1$-open). Observe that $\mathcal{G}_0$ is $C^1$ dense because $\La_\e$ is locally maximal with isolating neighborhood $U_\e$ for a $C^1$-typical $F$ (in view of Corollary~\ref{c.loc-max}). Assuming that $\mathcal{G}_{N-1}$ is $C^1$-dense for some $N\geq 1$, we claim that $\mathcal{G}_N$ is also $C^1$-dense. In fact, given $F\in\mathcal{G}_{N-1}\cap\mathcal{D}^2$ with the properties appearing in the proof of Corollary~\ref{c.loc-max}, we can refine the Markov partition $\{P_i\}_{i=1}^M$ so that $F^j(\{\pm\e\}\times \mathcal{K}_0)\cap P_i=\emptyset$ for every $0\leq j\leq N-1$. 

Next, for every $p\in\{\pm\e\}\times [0,1]$, we denote by $E(p)$ the tangent line of the $C^2$ curve $F^N(\{\pm\e\}\times [0,1])$ at the point $F^N(p)$. Note that $E(p)$ is a $C^1$ function of $p\in\{\pm\e\}\times [0,1]$. Therefore, since $\mathcal{K}_0$ is a $C^2$ dynamical Cantor set of Hausdorff dimension $HD(\mathcal{K}_0)<1$, we see that, without loss of generality, one can assume that the directions $E_i$ of the stable foliations (by parallel straight lines) of $P_i\cap W^s_{loc}(\hat{\La}_\e)$ don't belong to the set of directions $\{E(p): p\in\{\pm\epsilon\}\times \mathcal{K}_0\}$. Furthermore, by compactness, we can also fix a Markov partition $I_1,\dots, I_k$ of $\mathcal{K}_0$ of sufficiently small diameter so that the directions $E_i$ are still transversal to the finite collection of $C^2$ curves $F^N(\{\pm\e\}\times I_l)$ for every $i=1,\dots,M$ and $l=1,\dots,k$. At this stage, we write
$$P_i\cap F^N(\{\pm\e\}\times \mathcal{K}_0) = P_i\cap \bigcup\limits_{b=1}^{a(i)}F^N(\{\pm\e\}\times (\mathcal{K}_0\cap I_{l(b,i)}))$$
for an adequate choice of indices $l(b,i)\in\{1,\dots,k\}$, and we observe that, by transversality, the projection of each $F^N(\{\pm\e\}\times (I_{l(b,i)}\cap \mathcal{K}_0))$ along the direction $E_i$ gives a $C^{2}$ dynamical Cantor set $L_{b,i}$. Moreover, we note that $P_i\cap F^N(\{\pm\e\}\times \mathcal{K}_0)\cap W^s_{loc}(\hat{\La}_\e)\neq\emptyset$ if and only if $\mathcal{K}_i^s\cap(\cup_{b=1}^{a(i)}L_{b,i})\neq\emptyset$ (where $\mathcal{K}_i^s$ is the stable Cantor set introduced during the proof of Corollary~\ref{c.loc-max}). Using Moreira's theorem~\ref{gugu}, we obtain $(\widetilde{\mathcal{K}_i^s},\widetilde{f_i})$ dynamical Cantor sets $C^1$-close to $(\mathcal{K}_i^s,f_i)$ such that $\widetilde{\mathcal{K}_i^s}\cap(\cup_{b=1}^{a(i)}L_{b,i})=\emptyset$ for every $i$. It follows that $\widetilde{F}|P_i(x,y) := (\widetilde{f_i}(x),K_{sgn(x)}y)$ (in the linearizing coordinates inside each $P_i$) is $C^1$-close to $F\in\mathcal{G}_{N-1}$ and $\widetilde{\La}_\e\cap P_i = \widetilde{\mathcal{K}_i^s}\times \mathcal{K}_0$ (in the same linearizing coordinates). In particular, by construction, we have $\widetilde{F}\in\mathcal{G}_N$. This ends the argument.
\end{proof}

\begin{obs}In the previous statement, we deal with the returns to the critical strip of ``quasi-critical'' points $(\pm\e,y)$, $y\in \mathcal{K}_0$, instead of critical points $c^\pm_y$. The technical reason behind this procedure will be clear in the next section (when we perform the ``flatness'' perturbation to force critical points to fall into the basins of sinks).
\end{obs}

\subsection{Creating sinks whose large basins contain all critical points}\label{subpert}

\begin{lema}\label{flatness}For a $C^1$-open and dense subset of $F\in\mathcal{D}^1$, the critical points $c_y^{\pm}\in\{0^\pm\}\times \mathcal{K}_0$ belong to the union of the basins of a finite number of periodic sinks of $F$.
\end{lema}

\begin{proof}Fix $F\in\mathcal{D}^1$ be a $C^1$-generic map satisfying the properties of Lemma~\ref{return}. Given $\delta>0$, we will find a $C^1$-perturbation of $F$ with size $\delta$ whose critical points belong to the basins of finitely many periodic sinks. In this direction, we take $\e>0$ sufficiently small such that $|\partial_x f(x,y)|<\delta/2$ for every $|x|\leq\e$ and $y\in [0,1]$. Now, we perturb $F$ to make it ``flat'' in the critical strip $\overline{R_\e}:=[-\e,\e]\times [0,1]$, i.e., we define
$$g(x,y) = \left\{\begin{array}{rl}f(x,y) & \textrm{if }|x|\geq\e \\ f(\pm\e,y) & \textrm{if } |x|\leq\e\end{array}\right.$$
and $G(x,y):=(g(x,y),K_{sgn(x)}(y))$. Observe that, although $G\notin\mathcal{D}^1$ because $g(x,y)$ is not $C^1$, $G$ is $\delta/2$-close to $F$ in the Lipschitz norm and $G=F$ outside the critical strip $R_\e$. In particular, the pieces of orbits of $F$ and $G$ are equal while they stay outside $R_\e$. Hence, since $F$ satisfies Lemma~\ref{return}, we have that $G$ satisfies the same properties, namely, either its quasi-critical points $\{\pm\e\}\times \mathcal{K}_0$ return to the critical region $R_\e$ (after a bounded number of iterates) or they fall into the basins of finitely many periodic sinks (inside $\La_\e$). We claim that the quasi-critical points returning to $R_\e$ belong to the basins of finitely many  periodic \emph{sinks} of $G$. Indeed, by compactness and continuity, we can take a Markov partition $I_1,\dots,I_k$ of $K_0$ of small diameter and some integers $r_1,\dots,r_k$ so that every quasi-critical point $p\in\{\pm\e\}\times I_l$ return to $R_\e$ or fall into the basin of a sink after exactly $r_l$ iterates. Since the pieces of orbits of $F$ and $G$ outside $R_\e$ are the same, and the piece of the $G$-orbit outside $R_\e$ of a point $(x,y)\in R_\e$ equals to the piece of $F$-orbit outside $R_\e$ of the point $(\pm\e,y)$, we obtain that $G$ send the boxes $[-\e,\e]\times I_l$ strictly inside another (\emph{a priori} different box) $[-\e,\e]\times I_j$ or inside the basin of a periodic sink after $r_l$ iterates (exactly), so that our claim follows. Finally, we complete the proof by noticing that, although $G\notin\mathcal{D}^1$, one can \emph{slightly} ``undo'' the ``flat'' perturbation in order to get a $H\in\mathcal{D}^1$ such that its critical points belong to the basin of finitely many periodic sinks and $H$ is $\delta/2$-close to $G$ in the Lipschitz norm (and, \emph{a fortiori}, $H\in\mathcal{D}^1$ is $\delta$-close to $F\in\mathcal{D}^1$ in the $C^1$-topology).
\end{proof}

\subsection{End of the proof of Theorem B}

By combining Corollary~\ref{c.loc-max} and Lemma~\ref{flatness}, we get that the non-wandering set $\Omega(F)$ of a $C^1$-typical $F\in\mathcal{D}^2$ can be written as $\Omega(F)=\La_\e\cup\{p_1,\dots,p_k\}$ where $p_1,\dots,p_k$ are periodic sinks of $F$ whose (large) basins contain a $\e$-neighborhood of the critical set, i.e., $\Omega(F)$ is a hyperbolic set.

Thus,  the proof of Theorem B will be complete once we can show the following claim: a Kupka-Smale\footnote{In principle, we can not directly apply Kupka-Smale theorem in our context because $\mathcal{D}^1$ is a specific parameter space. However, it is not hard to see that the arguments in the proof of Kukpa-Smale theorem are still valid in our case. Indeed, the hyperbolicity of periodic points of a $C^1$-generic $F\in\mathcal{D}^1$ was verified in footnote 3, so it remains to check the generic transversality of invariant manifolds of periodic points, but this last task is ``automatic'' in our setting: we showed that the non-wandering set $\Omega(F)$ of a $C^1$-generic $F\in\mathcal{D}^1$ is the union of a hyperbolic set away from the critical line and a finite number of sinks whose basins cover the critical line. In particular, any point $q$ in the intersection of the stable and unstable manifolds of a hyperbolic periodic point $p$ of saddle-type does not approach the critical line and thus the horizontal direction $\mathbb{R}.(1,0)$ is invariant along the orbit of $q$. Hence, if the stable and unstable manifolds of $p$ meet non-transversely at $q$, we would deduce that the stable direction of $p$ is horizontal, a contradiction (as $p$ has saddle-type, so that the horizontal direction is unstable).} $F\in\mathcal{D}^1$ such that $\Omega(F)$ is hyperbolic is Axiom A. However, this is a consequence of the following argument of Pujals and Sambarino~\cite[p. 966]{PS1}: $\Omega(F)$ hyperbolic implies $L(F)$ hyperbolic, so that the results of Newhouse~\cite{N4} say that (1) periodic points are dense in $L(F)$ and we can do spectral decomposition of $L(F)$ into finitely many basic sets $L_1,\dots, L_k$, and (2) we have $\Omega(F)=L(F)$ whenever there is no cycle between the $L_i$. Hence, we can show that $\Omega(F)=L(F)$ whenever we can verify the no-cycles condition. Since our phase space is two-dimensional, a cycle can only occur among basic sets of saddle-type. However, since $F$ is Kupka-Smale, the intersections of invariant manifolds involved in this cycle are transversal, a contradiction.

\end{document}